\documentclass{article}
\usepackage[utf8]{inputenc}
\usepackage{amsfonts,amssymb,amsmath,amsthm}
\usepackage[left=2cm,right=2cm,top=2cm,bottom=2cm]{geometry}

\usepackage[utf8]{inputenc}
\usepackage{graphicx}
\usepackage{mathtools}

\usepackage{subcaption}
\usepackage{tikz}
\usepackage{enumerate}
\usepackage{cases}
\usepackage{multicol}
\usepackage{sidecap}
\usepackage{float}
\usepackage{xcolor}
\definecolor{blue2}{rgb}{0.67, 0.9, 0.93}
\usepackage[color=blue2]{todonotes}
\usepackage{cases}
\usepackage{extarrows}
\usepackage{soul}
\usepackage[shortlabels]{enumitem}

\usepackage[normalem]{ulem}
\usepackage[hidelinks=true]{hyperref}
\usepackage{setspace}
\numberwithin{equation}{section}
\newtheorem{theorem}{Theorem}[section]
\newtheorem{lemma}[theorem]{Lemma}

\theoremstyle{definition}

\newtheorem{corollary}[theorem]{Corollary}
\newtheorem{remark}[theorem]{Remark}

\newcommand{\N}{{\mathbb N}}
\newcommand{\R}{{\mathbb R}}

\newcommand{\Z}{\mathbb{Z}}

\newcommand{\To}{\rightarrow}

\usepackage{xspace}

\newcommand{\gM}{g_{\mathrm{MK}}}

\usepackage{authblk}
\title{\sc Explicit propagation reversal bounds for bistable differential equations on trees}
\author[1]{Petr Stehl\'{\i}k\thanks{\tt pstehlik@kma.zcu.cz}}
\affil[1]{\small Department of Mathematics and NTIS, Faculty of Applied Sciences, University of West Bohemia,\authorcr Univerzitn\'\i~8, 306 14 Plze\v{n}\\ Czech Republic}

\begin{document}

\maketitle

\begin{abstract}
In this paper we provide explicit description of the pinning region and propagation reversal phenomenon for the bistable reaction diffusion equation on regular biinfinite trees. In contrast to the general existence results for smooth bistabilities, the closed-form formulas are enabled by the choice of the piecewise linear McKean's caricature. We construct exact pinned waves and show their stability. The results are qualitatively similar to the propagation reversal results for smooth bistabilities. Major exception consists in the unboundedness of the pinning region in the case of the bistable McKean's caricature. Consequently, the propagation reversal also occurs for arbitrarily large diffusion.
\end{abstract}

\smallskip
\noindent\textbf{Keywords:} reaction-diffusion equations; lattice differential equations;
traveling waves; propagation reversal; propagation failure; homogeneous trees.

\smallskip
\noindent\textbf{MSC 2010:} 34A33, 37L60, 39A12, 65M22



\section{Introduction}
In this paper we provide explicit description of the propagation reversal and pinning region for the traveling waves of the reaction-diffusion equation on biinfinite $k$-ary trees
\begin{equation}\label{e:TDE}
    \dot{u}_i = d(ku_{i+1} - (k+1)u_i + u_{i-1} ) + g(u_i;a), \quad i\in \Z,\ t>0,
\end{equation}
with the bistable reaction function $g$ given by the piecewise linear McKean's caricature
\begin{equation}\label{e:McKean}
    \gM(u;a) = \begin{cases}
-u, & u< a,\\
1-u, & u\geq a.
\end{cases}
\end{equation}
Beside the closed-form description we show that the pinning region is unbounded in the $(d,a)$-plane which contrasts with the results for the propagation reversal with smooth bistabilities \cite{Hupkes2023}.

Reaction-diffusion equations play important role in chemistry, biology, ecology, material science, etc. They serve as the primary model for numerous dynamic phenomena including traveling waves or pattern formation. In the prevalent settings, the homogeneous continuous space leads to partial differential equations (PDEs) such as the scalar Nagumo equation
\begin{equation}\label{e:PDE}
    u_t = d u_{xx} + g(u;a), \quad x\in\R,\ t>0,
\end{equation} 
where $d>0$ is the diffusion parameter and $g$ is a bistable reaction function, e.g., the cubic
\begin{equation}\label{e:cubic}
    g(u;a) = u(1 -u)(u -a), \quad a\in (0,1).    
\end{equation} 
The traveling wave solution of the PDE~\eqref{e:PDE} has a form $u(x,t)=U(x-ct)$ with an increasing profile $U(z)$ satisfying $U(-\infty)=0$ and $U(+\infty)=1$ and nonzero speed $c\neq 0$. It exists whenever $a\neq 1/2$ in the case of the cubic $g$ given by \eqref{e:cubic}, see Figure~\ref{fig:pinnings}(a). In the case of a general bistable reaction $g$ on $[0,1]$ traveling waves occur whenever
\[
\int_0^1 g(u;a) \mathrm{d}u \neq 0.
\]

\begin{figure}
     \centering
     \begin{subfigure}[b]{0.4\textwidth}
         \centering
         \includegraphics[width=\textwidth]{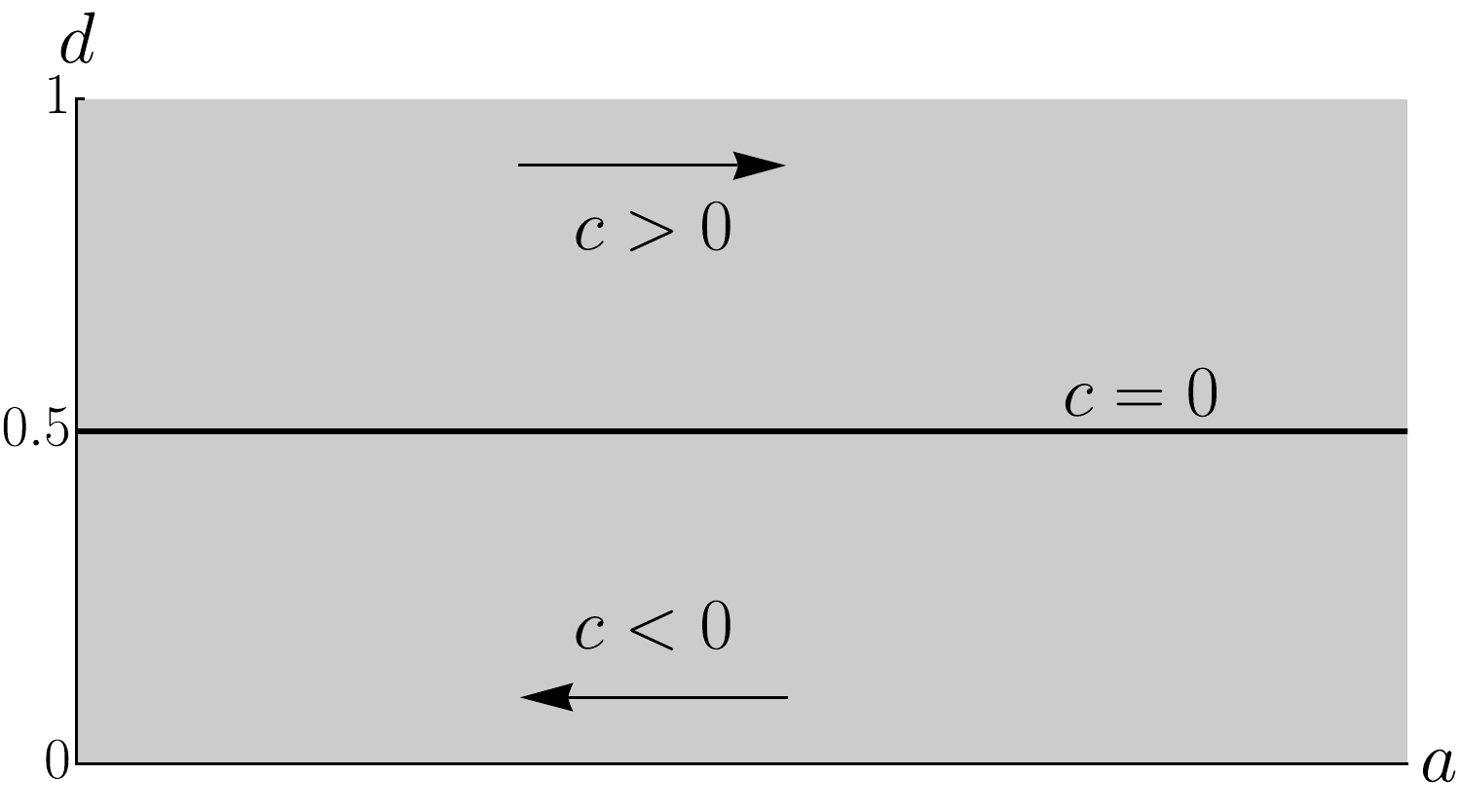}
         \caption{PDE \eqref{e:PDE}}         
     \end{subfigure}
          \begin{subfigure}[b]{0.4\textwidth}
         \centering
         \includegraphics[width=\textwidth]{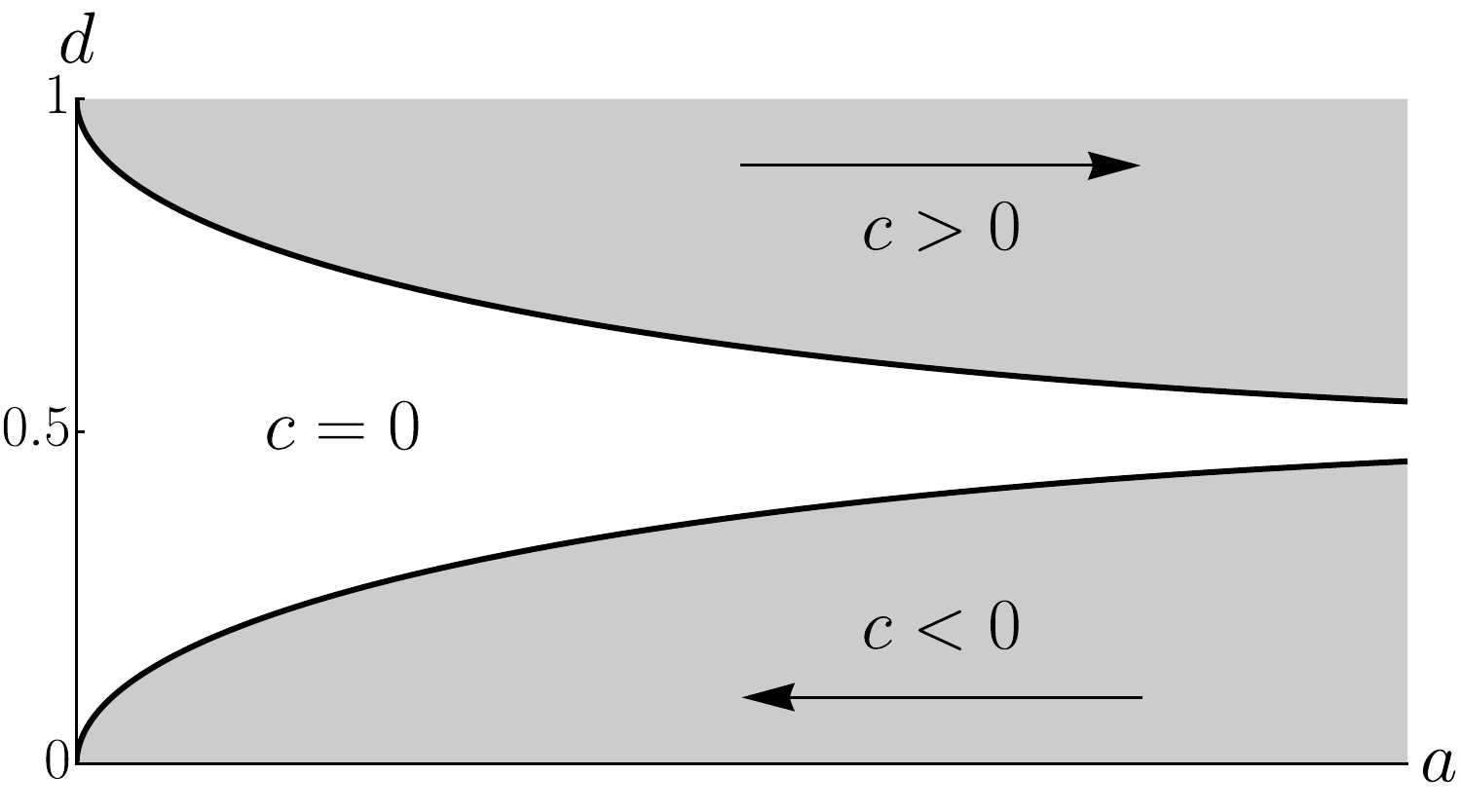}
         \caption{LDE \eqref{e:LDE}}         
     \end{subfigure}     
     \begin{subfigure}[b]{0.4\textwidth}
         \centering
         \includegraphics[width=\textwidth]{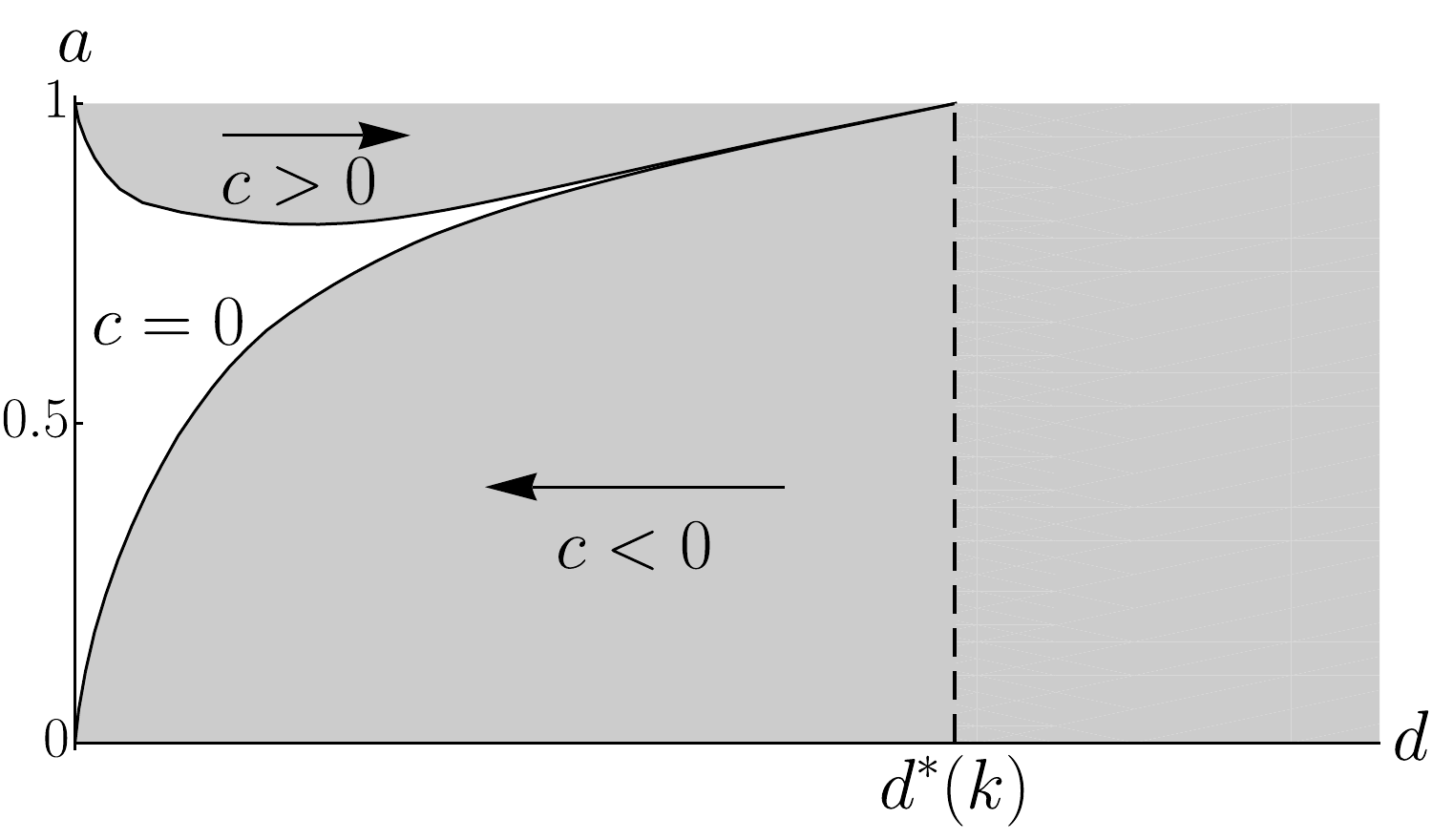}
         \caption{TDE \eqref{e:TDE}, smooth $g$}        
     \end{subfigure}
     \begin{subfigure}[b]{0.4\textwidth}
         \centering
         \includegraphics[width=\textwidth]{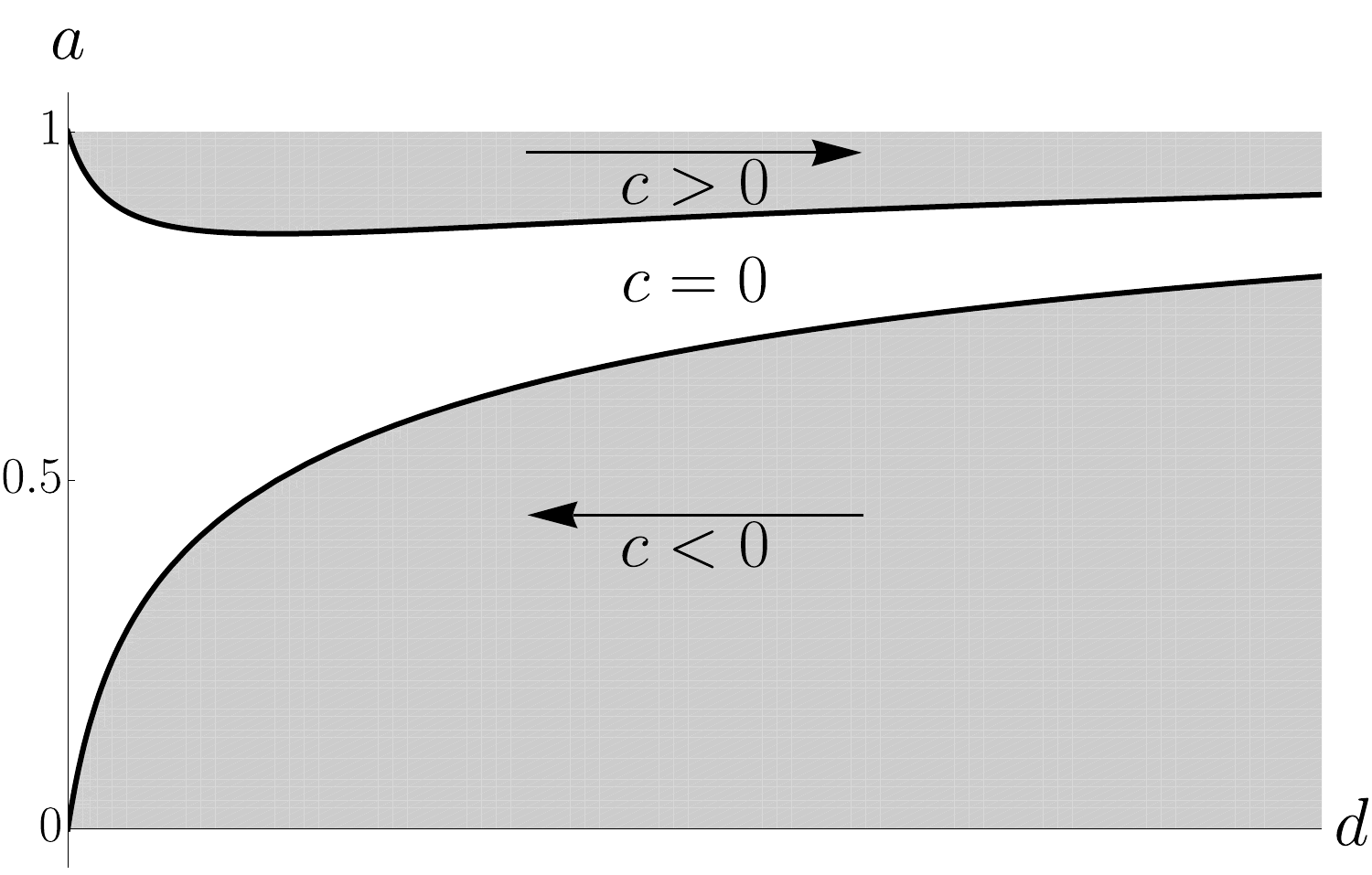}
         \caption{TDE \eqref{e:TDE}, McKean's $g$ given by \eqref{e:McKean}}    
    \end{subfigure}
        \caption{Wave directions and pinning regions for selected reaction-diffusion models.}
        \label{fig:pinnings}
\end{figure}

\paragraph{Lattice reaction-diffusion equations and propagation failure.} In many applications, the naturally discrete character of the underlying spatial structure (for instance neurons, cells, discrete habitats as islands, chemical mixers, \ldots) leads to discrete-space alternatives. In the simplest configuration, the lattice differential equation (LDE)
\begin{equation}\label{e:LDE}
\begin{aligned}
        \dot{u}_i & = d(u_{i+1} - 2u_i + u_{i-1} ) + g(u_i;a), \quad i\in \Z,\ t>0,
\end{aligned}
\end{equation}
may arise directly from the modeling of the natural spatial discreteness or from the PDE~\eqref{e:PDE} by the standard spatial Euler discretization. In contrast to the PDE~\eqref{e:PDE}, the traveling wave solution $u_i(t)=U(i-ct)$ with $c\neq 0$ exists only for sufficiently strong diffusion $d\gg 0$ for a given $a\in(0,1)$ for the LDE~\eqref{e:LDE} with the bistable cubic \eqref{e:cubic}, \cite{Zinner1991}. Moreover, it is pinned with $c=0$ for small values of the diffusion $d$ \cite{Keener1987}, see Figure~\ref{fig:pinnings}(b). This phenomenon is known as the \emph{propagation failure} and is characteristic to discrete-space models \cite{Chow1998}. The pinning region is the set of parameters  $(a,d)$ in which the monotone wave profiles do not move (i.e., $c=0$) and is characterized by the existence of large number of stationary patterns, a phenomenon known as the \emph{spatial topological chaos} \cite{Chow1998}.

\paragraph{Graph reaction-diffusion equations.} The homogeneous discrete structure of the integer lattice $\Z$ in LDE~\eqref{e:LDE} does not capture the intrinsic spatial heterogeneity of real-world networks of cells, neurons, islands, etc. \cite{Besse2021a, Slavik2020, Stehlik2017}. Considering a general graph $\mathcal{G}=(V,E)$ where $V$ is a set of vertices of $\mathcal{G}$ and $E$ its set of edges we arrive to the network alternative of the PDE~\eqref{e:PDE} and the LDE~\eqref{e:LDE}, the graph differential equation (GDE)
\begin{equation}\label{e:GDE}
    \dot{u}_i = d\sum_{j\in \mathcal{N}(i)} (u_j-u_i) + g(u_i;a),\quad i\in V,\ t\in\mathbb{R},
\end{equation}
where $\mathcal{N}(i)=\{j\in V: (i,j)\in E\}$ is the neighborhood of a vertex $i\in V$. GDEs have a rich set of stationary patterns \cite{Stehlik2017} whose properties also strongly depend on the graph structure. The bifurcation phenomena are far from being fully understood even on the simplest (small) graphs \cite{Stehlik2023a}. On the other hand, the description of the structure of stationary solutions in particular cases (such as the cyclic graphs $\mathcal{G}=\mathcal{C}_n$, \cite{Hupkes2019c}) helped to describe new type of nonmonotone traveling waves which exist in the pinning region of the LDE~\eqref{e:LDE}, \cite{Hupkes2019b}.

\paragraph{Tree reaction-diffusion equations.} The propagation phenomena on general graphs/networks seem to be an even harder challenge than the analysis of stationary patterns. Regular $k$-ary tree graphs consist of nodes which have a single `parent' and $k$ `children', with $k\in\N$, $k>1$. Tree graphs serve as a natural first step in the analysis of propagation on non-lattice graphs. Trees have been considered not only in the settings of the Nagumo equation \cite{Hupkes2024, Hupkes2023, Kouvaris2012} but also in the case of monostable Fisher equation \cite{Hoffman2019} and propagation in the SIR model \cite{Besse2021}. Whereas a tree may be considered finite with $|V|<\infty$ \cite{Kouvaris2012} or semiinfinite \cite{Besse2021, Hoffman2019}, we consider a biinfinite $k$-ary tree $\mathcal{T}_k=(V,E)$ as in \cite{Hupkes2024, Hupkes2023}. We number its vertices $V=\Z\times \N_0$. The set of edges $E$ (capturing connections of a parent $(i,j)$ to its $k$ children) is then (see Figure~\ref{fig:binary:tree})
\begin{equation*}
     \qquad E = \left\{\big((i,j), (i+1, kj + l)\big): \ i\in \Z, \ j\in \N_0, \ l\in \{0,\ldots, k-1\} \right\}.
\end{equation*}
We focus on a special class of solutions, the so called \emph{layer solutions}, \cite{Hupkes2023}, which have the same value for each `age generation' and for all $(i,j)\in V=\Z\times \N_0$ are independent of $j$, i.e.,
\[
u_{i,j}(t)=u_i(t).
\]
The layer solutions on biinfinite $k$-ary trees thus satisfy the simplified version of the GDE~\eqref{e:GDE}, the tree differential equation (TDE) \eqref{e:TDE}, in which $u_{i-1}(t)$ describes the value in the parent layer and $u_{i+1}$ the value in the children layer.

Beside the relative accessibility of the TDE~\eqref{e:TDE}, the natural motivation stems also from the fact that the lattice equation~\eqref{e:LDE} coincides with this model for $k=1$. Moreover, the results are also connected to the dynamics on random graphs, since $k$-ary trees locally approximate Erd\"{o}s-Rényi random graphs, \cite{Kouvaris2012}.

\paragraph{Lattice reaction diffusion advection equation.} Furthermore, the TDE~\eqref{e:TDE} can be also rewritten as the lattice reaction diffusion advection equation once we rearrange the terms into
\begin{equation}\label{e:kLDE}
        \dot{u}_i = d(u_{i+1} - 2u_i + u_{i-1}) + d(k-1)(u_{i+1} - u_i) +  g(u_i;a), \quad i\in \Z, \ t>0.
\end{equation}
Consequently, a general real advection parameter $k\in\R$, can be considered instead of the branching factor $k\in\N$, $k>1$. We restrict our attention to $k>1$ so that the advection direction coincides.

While studying the traveling waves solutions of the TDE~\eqref{e:kLDE} we are looking for solutions of the form $u_i(t)=U(i-ct)$ which connect the two stable equilibria of the bistability $g$, i.e., $U(-\infty)=0$ and $U(+\infty)=1$. Waves with $c>0$ can be seen as right-moving from the lattice formulation~\eqref{e:kLDE} but also down the tree (from parents to children) from the tree point of view, see Figure~\ref{fig:binary:tree}. Similarly, waves with $c<0$ are left moving from the lattice point of view but can be seen as moving up the tree or from children to parents in the interpretation of tree layers.


\begin{figure}
\centering
\begin{minipage}{.5\textwidth}
  \centering
  \includegraphics[width=\linewidth]{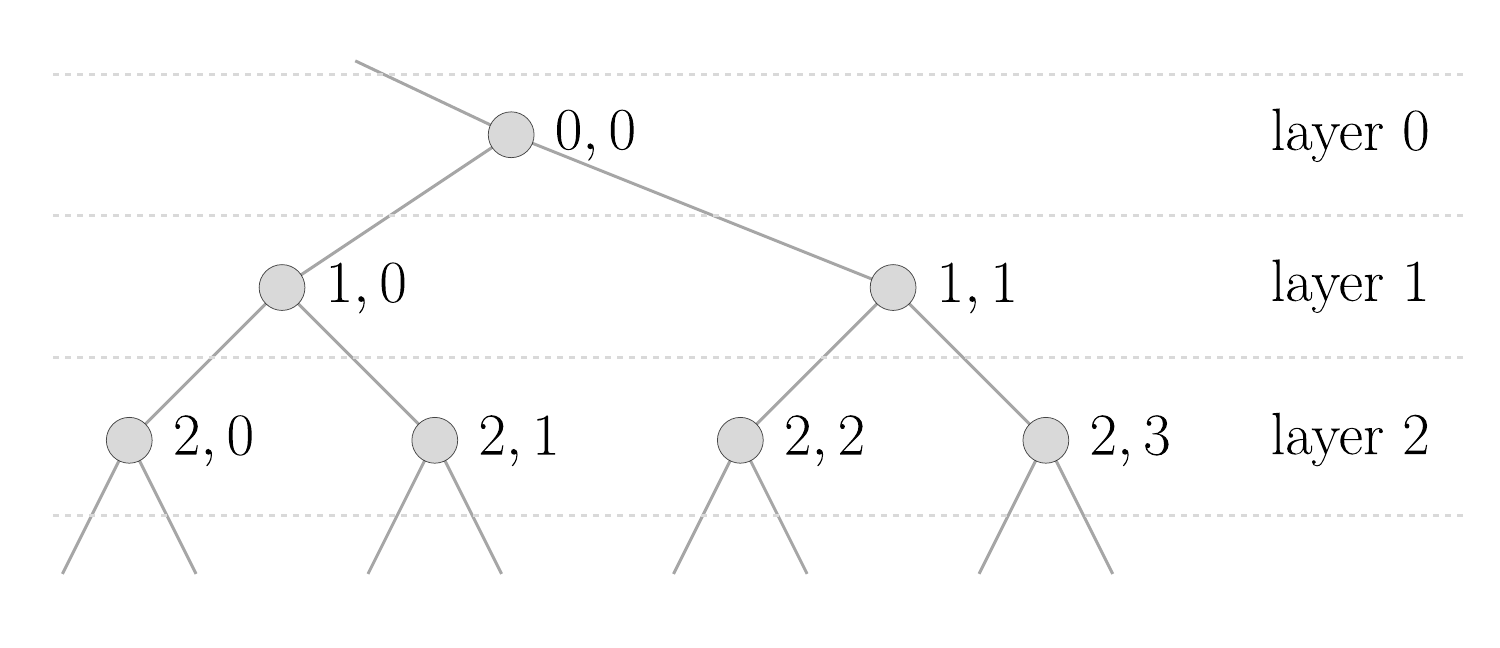}
  \captionof{figure}{Part of a biinfinite binary tree $\mathcal{T}_2$ and illustration of its layers.}
  \label{fig:binary:tree}
\end{minipage}\phantom{aa}
\begin{minipage}{.47\textwidth}
  \centering
  \includegraphics[width=.7\linewidth]{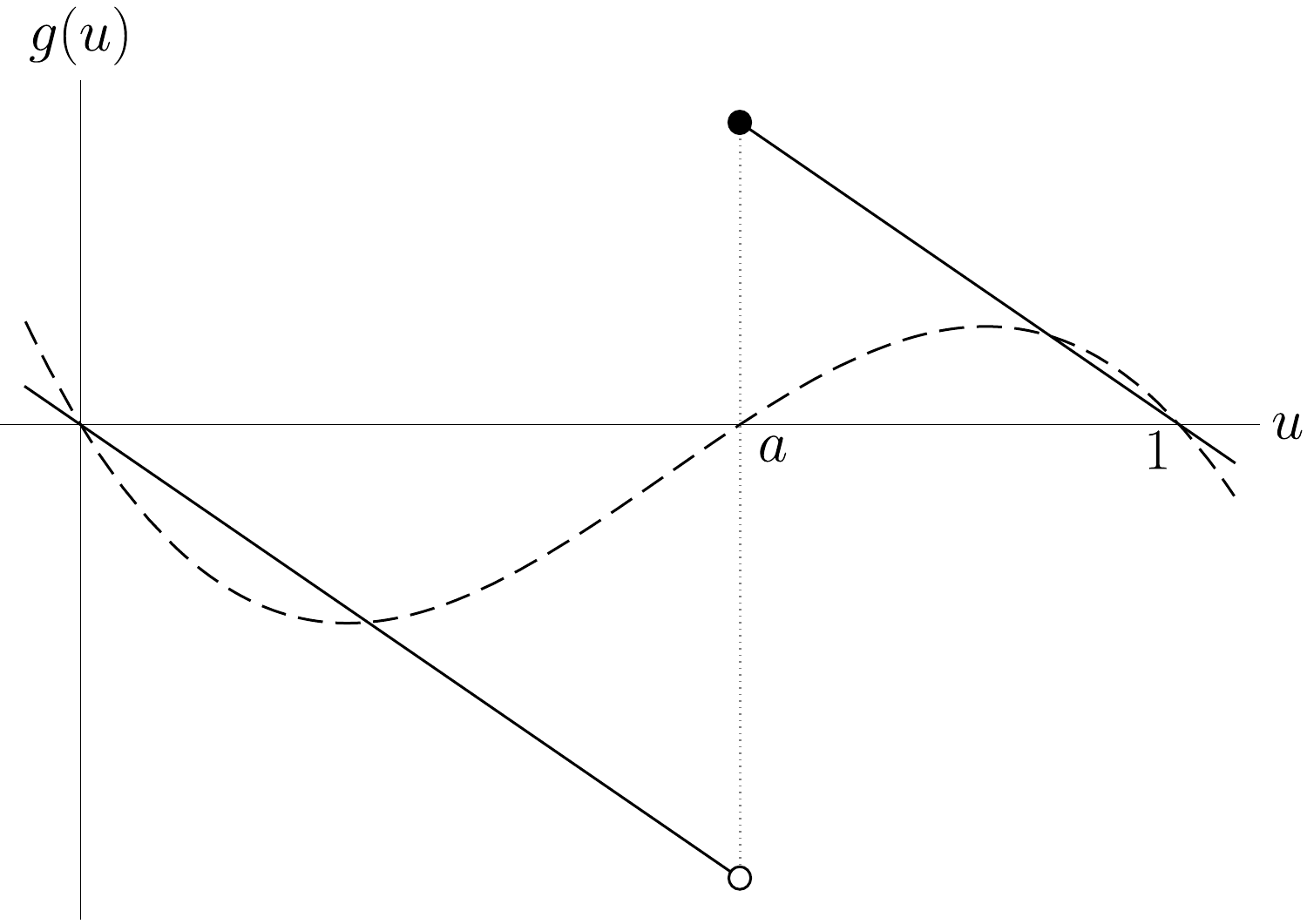}
  \captionof{figure}{A cubic bistability (dashed) and McKean's caricature \eqref{e:McKean} (solid).}
  \label{fig:bistabilities}
\end{minipage}
\end{figure}

\paragraph{Propagation reversal.} We considered the TDE~\eqref{e:TDE} with a general smooth bistability $g$ in \cite{Hupkes2023} and constructed sub- and supersolutions to show that for any $k>1$ there exist only bounded sets of parameters $(d,a)$ for which the waves are pinned ($c=0$) or propagate down the tree with $c>0$. Moreover, there is a value $d^*(k)$ such that the waves propagate up the tree with $c<0$ for a given $k>1$ and all $a\in(0,1)$, see Figure~\ref{fig:pinnings}(c).

In particular, these properties describe the \emph{propagation reversal} of waves for any $k>1$ and sufficiently large $a\approx 1$. This phenomenon implies the existence of values $0<\underline{d}_+(a,k)<\overline{d}_+(a,k)<\underline{d}_-(a,k)<d^*(k)$. Fixing $a\approx 1$ and increasing $d$, the waves are pinned for $d\in(0,\underline{d}_+(a,k))$, move down the tree with $c>0$ for $d\in(\underline{d}_+(a,k),\overline{d}_+(a,k))$, are pinned again with $c=0$ for $d\in(\overline{d}_+(a,k),\underline{d}_-(a,k))$, and reverse the direction and move up the tree with $c<0$ for $d>\underline{d}_-(a,k)$.

In \cite{Hupkes2024} the spectral convergence method inspired by \cite{Bates2003} was used to further understand and describe in detail the propagation reversal for the general smooth reaction functions $g$ in the large (but finite) diffusion regime $d\approx d^*(k)$ near $a\approx 1$, see Figure~\ref{fig:pinnings}(c).


\paragraph{McKean's caricature.} The goal of this paper is to study the propagation reversal phenomenon for the TDE~\eqref{e:TDE} with the McKean's bistable caricature (see Figure~\ref{fig:bistabilities}) $\gM$ given by \eqref{e:McKean} and get explicit bounds on the pinning region and propagation reversal. Since this piecewise linear bistability is not smooth, the results from \cite{Hupkes2023} are not applicable. McKean's caricature have been used in various alternative forms. Using the Heaviside function $H(x)$ we can write \eqref{e:McKean} as 
\[\gM(u;a)=-u + H(u-a). \]
Alternatively, multivalued versions could be considered, e.g.,
\begin{equation}\label{e:McKean:multi}    
    \gM^*(u;a) = \begin{cases}
-u & u< a,\\
[-a,1-a] & u=a,\\
1-u & u\geq a.
\end{cases}
\end{equation}

The bistable caricature \eqref{e:McKean} was proposed in \cite{McKean1970} and has been widely used. In particular, F\'{a}th \cite{Fath1998} obtained explicit bounds for the pinning region of LDE~\eqref{e:LDE} with McKean's bistability which are qualitatively the same as those for the LDE with the bistable cubic, see Figure~\ref{fig:pinnings}(b). Further F\'{a}th applied similar ideas  to describe bifurcations for the semiinfinite lattice in \cite{Fath1999}. The caricature was also used in other settings, for instance in the antidiffusion regime \cite{Vainchtein2009}, the nonhomogeneous diffusion case \cite{Moore2014}, or FitzHugh-Nagumo system \cite{Tonnelier2003}.

\paragraph{Main result.}



In this paper we construct explicit pinned waves of the TDE~\eqref{e:TDE} with the McKean's bistable caricature $g=\gM$ from \eqref{e:McKean} and derive exact bounds
\begin{equation}\label{e:bounds}
a_\pm(d,k)=\frac{1}{2}\left(1+\frac{(k-1)d\pm 1}{\sqrt{\left(1+d(k + 1)\right)^2-4 d^2 k}} \right),
\end{equation}
which characterize the region where the stationary waves exist in the following way.

\begin{theorem}\label{t:main}
Let $g=\gM$ be given by \eqref{e:McKean} and $d>0$. Then there exists a pinned monotone wave of the TDE~\eqref{e:TDE} if and only if the parameter $a$ satisfies
\begin{equation}\label{e:main}
    a_-(d,k) < a \leq a_+(d,k).
\end{equation}
Moreover, the waves are stable.
\end{theorem}

In particular, Theorem~\eqref{e:main} implies exact bounds characterizing the propagation reversal for the TDE~\eqref{e:TDE} with the McKean's bistability~\eqref{e:McKean}. The properties are very similar to those obtained for the TDE~\eqref{e:TDE} with smooth bistabilities \cite{Hupkes2023}. The major difference consists in the fact that the pinning region is no longer bounded. Consequently, propagation reversal and the change of wave directions occur for arbitrarily large $d$ and $a\approx 1$, compare panels (c)--(d) in Figure~\ref{fig:pinnings}.

\begin{figure}
\begin{minipage}[c]{\linewidth}
\includegraphics[width=.44\linewidth]{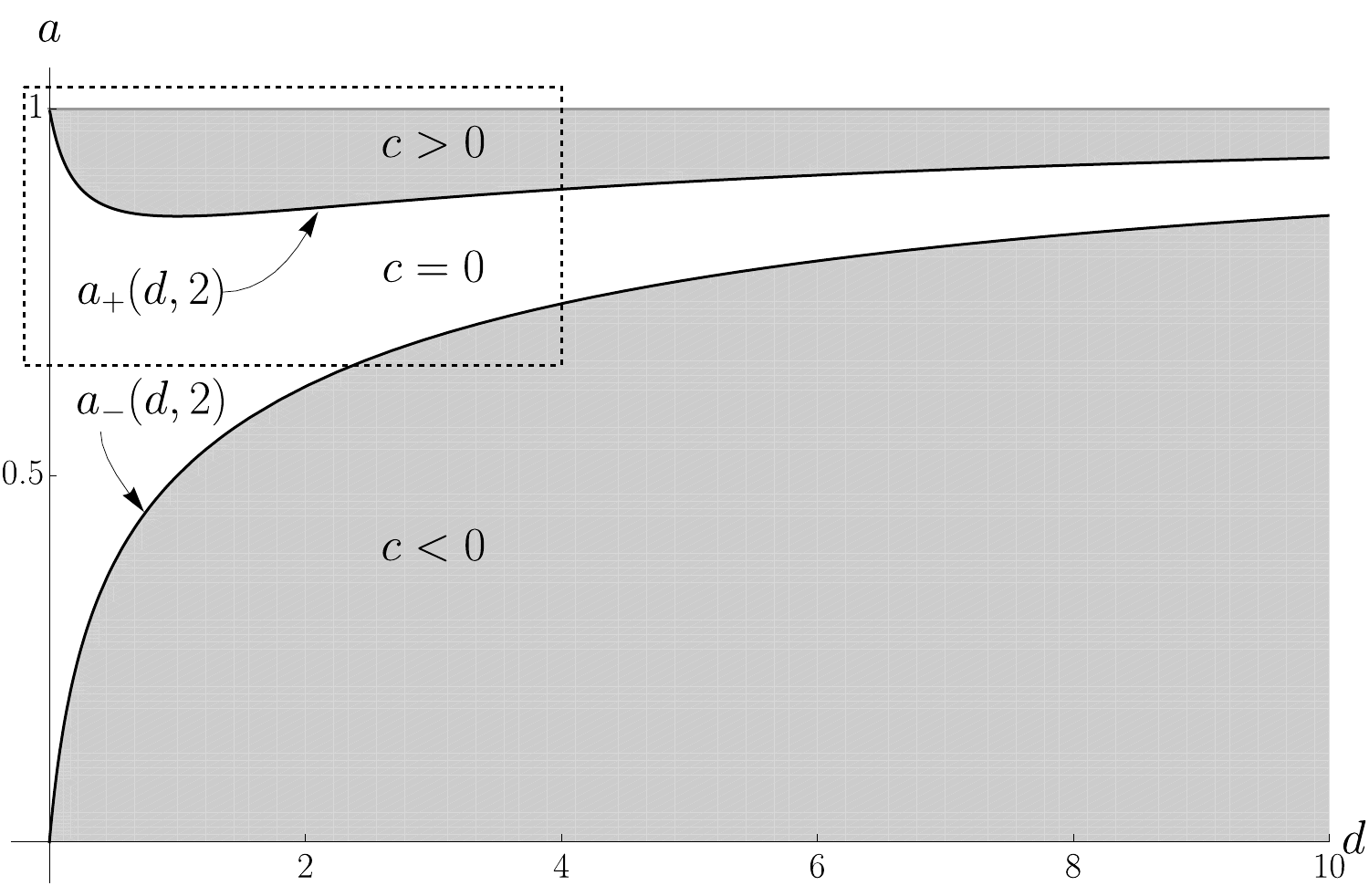}
\includegraphics[width=.54\linewidth]{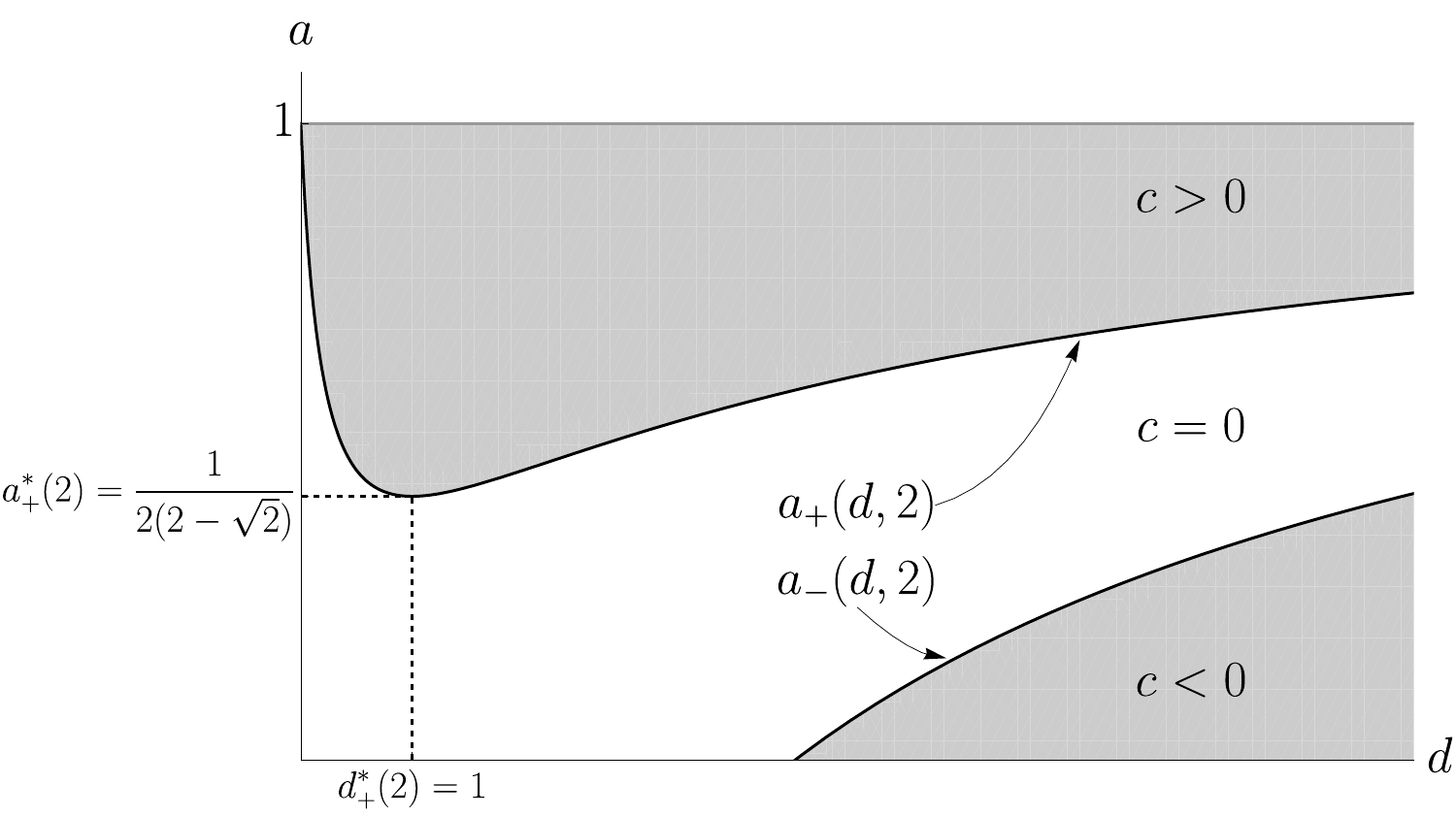}
\caption{Illustration of Theorem~\ref{t:main} $k=2$. The right panel represents the detail of the dashed rectangle in the left one.}\label{fig:k=2}
\end{minipage}
\end{figure}

\paragraph{Paper structure.} In Section~\ref{sec:proof} we provide a constructive proof of Theorem~\ref{t:main}. In Section~\ref{sec:properties} we focus on the properties of the bounds $a_\pm(d,k)$ and compare the results \cite{Fath1998} to relevant papers for the LDE~\eqref{e:LDE} and those for the TDE~\eqref{e:TDE} with smooth reactions \cite{Hupkes2024, Hupkes2023}.

\section{Proof of Theorem~\ref{t:main}}\label{sec:proof}
In order to prove Theorem~\ref{t:main} we construct a pinned wave. A stationary solution of \eqref{e:TDE} satisfies
\begin{equation}\label{e:stat}
0 = d(ku_{i+1} - (k+1)u_i + u_{i-1} ) + g(u_i;a), \quad i\in \Z. \tag{S}
\end{equation}
Taking an arbitrary horizontal shift into account, we assume, without loss of generality, that the double sequence $(u_i)$ is an increasing profile with 
\begin{equation}\label{e:limits}
\begin{split}
 u_i\rightarrow 0 &\text{ as } i\rightarrow-\infty,\\
 u_i\rightarrow 1 &\text{ as } i\rightarrow\infty,
\end{split}    
\end{equation}
and
 \begin{equation}\label{e:values}
\begin{split}
 u_i < a &\text{ as } i=-1,-2,-3,\ldots\\
 u_i\geq a &\text{ as } i=0, 1,2,3,\ldots
\end{split}    
\end{equation}
 
\noindent Employing the definition of $\gM$ \eqref{e:McKean}, the stationary solutions \eqref{e:stat} satisfy one of two linear difference equations, either the homogeneous one if $u_i<a$
\begin{equation}\label{e:stat-}
0 = d(ku_{i+1} - (k+1)u_i + u_{i-1} ) -u_i=dku_{i+1} - (d(k+1)+1)u_i + d u_{i-1}, \quad i=-1,-2,-3,\ldots, \tag{S$_-$}
\end{equation}
or the nonhomogeneous one if if $u_i\geq a$
\begin{equation}\label{e:stat+}
0 = d(ku_{i+1} - (k+1)u_i + u_{i-1} )+ 1-u_i=dku_{i+1} - (d(k+1)+1)u_i + d u_{i-1}+1, \quad i=0,1,2,3,\ldots \tag{S$^+$}
\end{equation}

The characteristic equation of both linear difference equations \eqref{e:stat-}, \eqref{e:stat+} 
\[
dk\lambda^2 - (d(k + 1) + 1)\lambda + d=0
\]
implies that
\begin{equation}
    \lambda_{12}=\frac{1+d(k + 1)\pm\sqrt{\left(1+d(k + 1)\right)^2-4 d^2 k}}{2 d k},
\end{equation}
with $\lambda_1>1$ and $\lambda_2\in(0,1)$ for $d>0$ and $k>1$.

The limit conditions \eqref{e:limits} and the equations \eqref{e:stat-}, \eqref{e:stat+}   imply that
\begin{equation}
    u_i=\begin{cases}
        C\cdot\lambda_1^i \quad \text{if } i=0,-1,-2,\ldots,\\
        1-D\cdot\lambda_2^i \quad \text{if } i=-1,0,1,2,\ldots\\
    \end{cases}
\end{equation}
Since both branches coincide for $i=0$ and $i=-1$ we get
\[
C=1-D, \text{ and } \frac{C}{\lambda_1}=1-\frac{D}{\lambda_2},
\]
which yields that
\[
C=\frac{1-\lambda_2}{\lambda_1-\lambda_2}\cdot \lambda_1,\quad D=\frac{\lambda_1-1}{\lambda_1-\lambda_2}\cdot\lambda_2.
\]
Thus the explicit solution is given by
\begin{equation}\label{e:pinned:full}
    u_i=\begin{cases}
        \frac{1-\lambda_2}{\lambda_1-\lambda_2}\cdot\lambda_1^{i+1} \quad \text{if } i=0,-1,-2,\ldots,\\
        1-\frac{\lambda_1-1}{\lambda_1-\lambda_2}\cdot\lambda_2^{i+1} \quad \text{if } i=-1,0,1,2,\ldots\\
    \end{cases}
\end{equation}
Specifically, for $i=0$ and $i=-1$ we have that
\begin{align*}
u_0 &=  \frac{1-\lambda_2}{\lambda_1-\lambda_2}\cdot\lambda_1 = \frac{1}{2}\left(1+\frac{(k-1)d+1}{\sqrt{\left(1+d(k + 1)\right)^2-4 d^2 k}} \right), \\
u_{-1} &=  \frac{1-\lambda_2}{\lambda_1-\lambda_2} = \frac{1}{2}\left(1+\frac{(k-1)d- 1}{\sqrt{\left(1+d(k + 1)\right)^2-4 d^2 k}} \right),
\end{align*}
i.e., $u_0=a_+(d,k)$ and $u_{-1}=a_-(d,k)$. Employing the Ansatz~\eqref{e:values} we obtain the inequalities \eqref{e:main}. Conversely, if \eqref{e:main} holds we can use the construction above to get the pinned wave~\eqref{e:pinned:full}.

\begin{remark}\label{r:nonstrict}
    If we repeat the construction with the multivalued caricature~\eqref{e:McKean:multi}, we obtain~\eqref{e:main} with both inequalities being nonstrict.
\end{remark}

The following lemma shows that the pinned waves are stable.
\begin{lemma}\label{l:stability}
    Let $g$ be given by \eqref{e:McKean}, $d>0$, $k>1$, and $a$ satisfy \eqref{e:main}. Then the pinned wave given by~\eqref{e:pinned:full} is stable.
\end{lemma}
\begin{proof}
    Let $u_i$ be the pinned wave constructed in \eqref{e:pinned:full}. We consider its perturbation
    \begin{equation}\label{e:perturb}
        \tilde{u}_i=u_i-p_i,\quad i\in \mathbb{Z}
    \end{equation}
    and assume, without loss of generality, that $p_i>0$. This is possible by employing a horizontal shift $u_{i+k}$, $k\in\mathbb{Z}$. Furthermore, we also assume that $p_i$ is small enough so that if $u_i>a$ then also $\tilde{u}_i>a$. 

    Consequently, we can substitute \eqref{e:perturb} into the TDE~\eqref{e:TDE} and take advantage of the fact that $u_i$ satisfies \eqref{e:stat} to obtain a linear LDE
    \begin{align}
        \dot{p}_i(t)&=d \left( k {p}_{i+1}(t)-(k+1)p_i(t)+p_{i-1}(t) \right)-p_i(t), \quad i\in\mathbb{Z},\ t>0,\notag \\
        &= d k {p}_{i+1}(t)-\left(d(k+1)+1\right)p_i(t)+dp_{i-1}(t).\notag
    \end{align}
    There is an explicit solution for the unique bounded solution $\hat{p}(t)$ for its Gaussian kernel satisfying
    \[
    \hat{p}_0(0)=1,\ \hat{p}_i(0)=0,\quad i\in\mathbb{Z}\setminus\{0\}.
    \]
    given by \cite[Section~4.2]{Slavik2014}
    \[
    \hat{p}_i(t)=\frac{e^{-\left(d(k+1)+1\right)t} \cdot I_x(2d\sqrt{k}t)}{\sqrt{k^i}},
    \]
    where $I_x$ is the modified Bessel function of the first kind. Since for $t\rightarrow\infty$ we have \cite[Theorem~4.5]{Slavik2014}
    \[
    \hat{p}_i(t)\sim\frac{e^{-\left(2d\sqrt{k}-d(k+1)-1\right)t}}{\sqrt{4\pi d\sqrt{k}k^it}},
    \]
    we have that both $\hat{p}_i(t)\rightarrow 0$ and ${p}_i(t)\rightarrow 0$ as $t\rightarrow\infty$, see \cite{Slavik2014}.
\end{proof}

\section{Properties of the pinning region}\label{sec:properties}
In order to describe differences with the TDE~\eqref{e:TDE} with smooth $g$, let us study in detail the bounds $a_\pm(d,k)$ from \eqref{e:bounds}.
\begin{theorem}\label{t:properties}
    Let $d>0$ and $k>1$. Then the bounds $a_\pm(d,k)$ defined in \eqref{e:bounds} satisfy
    \begin{enumerate}[(a)]
        \item $\lim\limits_{d\To 0+} a_-(d,k)= 0 $ for all $k>1$,
        \item $\lim\limits_{d\To 0+} a_+(d,k)= 1 $ for all $k>1$,
        \item $\lim\limits_{d\To\infty} a_\pm(d,k)= 1 $ for all $k>1$,
        \item $\lim\limits_{k\To 1+} a_\pm(d,k)=\frac{1}{2}\left(1\pm\frac{1}{\sqrt{1+4d}}\right)$ for all $d>0$,
        \item $\lim\limits_{k\To \infty} a_\pm(d,k)=1$ for all $d>0$,
        \item for a given $k>1$, the bound $a_-(d,k)$ is strictly increasing for $d>0$,
        \item for a given $k>1$, the bound $a_+(d,k)$ attains a global minimum 
        \begin{equation}\label{e:a+k:minimum}
            a_+^*(k)=\frac{1}{2\left(k-\sqrt{k(k-1)}\right)}
        \end{equation} at $d_+^*(k)=\frac{1}{k-1}$, is strictly decreasing for $0<d<d_+^*(k)$ and strictly increasing for $d>d_+^*(k)$,
        \item for a given $d>0$ the bounds $a_\pm(d,k)$ are strictly increasing for $k>1$.
    \end{enumerate}    
\end{theorem}
\begin{proof}
Because of the explicit nature of Theorem~\ref{e:main}, the proofs are rather straightforward application of standard calculus and thus omitted.
\end{proof}
\begin{remark} Note that the item (d) is consistent with the result of Fáth \cite{Fath1998} for the LDE~\eqref{e:LDE} with McKean's bistability~\eqref{e:McKean}. Similarly, other properties mimic roughly those for the TDE~\eqref{e:TDE} with smooth bistabilities. However, the expressions are based on explicit computations (e.g., explicit values in (g)) and the main qualitative  difference is connected to the item (c): the bound $a_+$ exist for all $d>0$ which is not the case for smooth nonlinearities, see \cite[Theorem 2.5]{Hupkes2023}, see Figure~\ref{fig:pinnings}.
\end{remark}

We can also formulate this result in terms of bounds on the diffusion $d$. Because of the nonmonotonicity of $a_+$ (see Figure~\ref{fig:k=2} and Theorem~\ref{t:properties}(g)) we need three bounds in this case.

\begin{corollary}[Propagation reversal]\label{t:propagation:reversal}
Let $k>1$ and $a>a_+^*(k)$ where $a_+^*(k)$ is defined by \eqref{e:a+k:minimum}. Then there exist values $\underline{d}_+(a,k)$, $\overline{d}_+(a,k)$, $\underline{d}_-(a,k)$ such that:
\begin{enumerate}[label=(\alph*)]
    \item there is a stationary wave for $d\leq \underline{d}_+(a,k)$,
    \item there is a traveling wave with $c>0$ for $\underline{d}_+(a,k) < d < \overline{d}_+(a,k)$,    
    \item there is a stationary wave for $\overline{d}_+(a,k)\leq d<\underline{d}_-(a,k)$,
    \item there is a traveling wave with $c<0$ for $d>\overline{d}_+(a,k)$.    
\end{enumerate}
\end{corollary}
\begin{remark}
    Closed-form expressions for $\underline{d}_+(a,k)$, $\overline{d}_+(a,k)$, $\underline{d}_-(a,k)$ can be obtained from equations $a_\pm(d,k)=0$, employing \eqref{e:bounds}. 
    
    Furthermore, in the spirit of Remark~\ref{r:nonstrict} all inequalities become nonstrict once the multivalued bistability~\eqref{e:McKean:multi} is considered.    
\end{remark}

\paragraph{Acknowledgments} PS gratefully acknowledges the support by the Czech Science Foundation grant no. GA22-18261S.

\end{document}